\documentclass{article}
\usepackage{xcolor}
\definecolor{hsr}{HTML}{D81B60}

\usepackage{amsmath}
\usepackage{amsfonts}
\usepackage{amssymb}
\usepackage{amsthm}
\usepackage{appendix}
\usepackage{array}
\usepackage{stmaryrd}
\usepackage{tikz-cd, tikz}
\usepackage{graphicx}
\usepackage[margin=1.5in]{geometry}
\usepackage{thmtools}
\usepackage{thm-restate}
\allowdisplaybreaks
\usepackage{hyperref}
\usepackage{longtable}
\hypersetup{
    colorlinks,
    linkcolor=hsr,
    citecolor=cyan,
    urlcolor=pink
}
\usepackage{cleveref}

\newtheorem{theorem}{Theorem}[section]
\newtheorem*{theorem*}{Theorem}
\newtheorem{lemma}[theorem]{Lemma}
\newtheorem*{lemma*}{Lemma}

\newtheorem*{proposition*}{Proposition}
\newtheorem{corollary}[theorem]{Corollary}
\theoremstyle{definition}
\newtheorem{definition}[theorem]{Definition}
\theoremstyle{remark}
\newtheorem{remark}[theorem]{Remark}

\newcommand{\Z}{\mathbf{Z}}

\usepackage[backend=bibtex]{biblatex}
\title{Anagrammatic quotients of free groups}
\author{Eric Stubley}
\date{\today}

\begin{document}
\maketitle
\begin{abstract}
We determine the structure of the quotient of the free group on 26 generators by English language anagrams.
This group admits a surprisingly simple presentation as a quotient of the free group by 301 of the possible 325 commutators of pairs of generators; all of the 24 missing commutators involve at least one of the letters $j, q, x, z$.
We describe the algorithm which can be used to determine this group given any dictionary, and provide examples from the SOWPODS scrabble dictionary witnessing the 301 commutators found.
\end{abstract}
\tableofcontents


\section{Introduction}

In this article we study the structure of the group
\[
A = \langle a, b, c, \ldots, z \, | \, w_{1} = w_{2} \text{ for all pairs $w_{i}$ which are English language anagrams}\rangle.
\]
This work was inspired by the classic article \cite{homophonic} by Mestre--Schoof--Washington--Zagier determining the structure of the homophonic group $H$.
The group $H$ has a similar definition to $A$ except anagrams are replaced by homophones.
The main result of \cite{homophonic} is that $H$ is trivial (in both English and French!); moreover it can reasonably be said that $H$ is independent of the dictionary chosen, in that the words witnessing the triviality of each generator should belong to anything calling itself an English language dictionary.
A similar study of homophonic groups for German, Korean, and Turkish was carried out in \cite{homophonic_2}.

The group $A$ is not trivial, and is not independent of the dictionary chosen.
In the particular case of the SOWPODS scrabble dictionary \cite{sowpods}, we prove the following.

\begin{theorem}\label{thm:intro}
The group $A$ with respect to the SOWPODS scrabble dictionary is the quotient of the free group on the 26 generators $a$, $b$, \ldots, $z$, subject to the relations $[\alpha, \beta] = 1$ for each pair of generators except the following 24:
\begin{itemize}
\item the 6 commutators of each pair of $j, q, x, z$,
\item the 5 commutators of $j$ with $f, k, l, w, y$,
\item the 6 commutators of $q$ with $b, f, g, k, w, y$,
\item the 3 commutators of $x$ with $f, k, v$,
\item the 4 commutators of $z$ with $f, k, v, w$.
\end{itemize}
\end{theorem}

Notably the relations needed are no more complicated than commutators of the generating letters.
This means that $A$ (with this dictionary) is a right-angled Artin group!

Some commutator relation are directly witnessed by anagram pairs.
For example the anagram relation $able = bale$ is nothing other than the commutator relation $[a, b] = 1$.
However, not all of the commutator relations in $A$ that appear in \cref{thm:intro} are directly obtainable from anagrams in this way.
In \cref{sec:algebra} we describe how some anagram relations can be simplified in the presence of commutators, in order that more commutators may be exhibited as coming from anagram relations.
\Cref{sec:strategy} describes the algorithm used to search the dictionary for commutators.
For the specific case of the SOWPODS scrabble dictionary we discuss our findings in \cref{sec:dictionary}.
As it turns out the author has a surprising personal connection to the history of the study of the group $A$, which is recounted in \cref{sec:history}.
\Cref{app:data} collects anagram pairs in the SOWPODS dictionary which exhibit the commutators in $A(\text{SOWPODS})$.

To avoid confusion, we reserve the use of lower case roman characters $a, b, c, \ldots, z$ in mathematical statements for the generators of our anagrammatic groups.
We will use greek characters ($\alpha, \beta, \ldots$), upper case roman characters ($A, B, \ldots$), or roman characters with subscripts ($a_{1}, b_{2}$, \ldots) for variables.

\section{Algebraic preliminaries}\label{sec:algebra}

We will work throughout this article with the free group $F_{26}$ with generators $a, b, \ldots, z$.

\begin{definition}
If we are given a word $W = a_{1}\ldots a_{n} \in F_{26}$ and $\sigma \in S_{n}$, we define $W_{\sigma}$ to be the word
\[
W_{\sigma} = a_{\sigma(1)}\ldots a_{\sigma(n)}.
\]
For two words $W, W' \in F_{26}$ we say that $W$ and $W'$ are anagrams of one another if $W' = W_{\sigma}$ for some permutation $\sigma$.
\end{definition}

\begin{definition}
Given a dictionary $D$ which consists of words in the generators $a, b, \ldots, z$ we define to be $R(D)$ the normal subgroup generated by all elements $W_{1}W_{2}^{-1}$ where $W_{i} \in D$ and $W_{1}, W_{2}$ are anagrams of one another.
We say that the relation $W_{1} = W_{2}$ is in $R(D)$ if $W_{1}W_{2}^{-1} \in R(D)$.
We define the anagram group $A(D)$ for the dictionary $D$ to be the quotient $A(D) = F_{26}/R(D)$.
\end{definition}

The following lemma shows that the ``smallest'' we can expect $A(D)$ to be is $\Z^{26}$.
This is perhaps clear if we think about the fact that two words are anagrams if and only if they are written with the same multiset of letters, which is precisely what $\Z^{26}$ counts.
Nonetheless we provide an explicit proof by computations with commutators to get used to thinking about commutators in an anagrammatic context.

\begin{lemma}
For any dictionary $D$, the abelianization map $F_{26} \to \Z^{26}$ factors through $A(D)$.
\end{lemma}
\begin{proof}
To show that $A(D)$ is intermediate between $F_{26}$ and $\Z^{26}$, it suffices to show that $R(D) \subseteq [F_{26}, F_{26}]$.
Since the commutator subgroup is normal, it suffices to show that any anagram relation $W = W'$ of $R(D)$ is in $[F_{26}, F_{26}]$.

Suppose that $\sigma$ is the permutation with $W' = W_{\sigma}$.
If we write $\sigma = \prod_{i=1}^{k} \tau_{i}$ as a product of transpositions, we see that
\[
WW_{\sigma}^{-1} = (WW_{\tau_{1}}^{-1})(W_{\tau_{1}}W_{\tau_{1}\tau_{2}}^{-1}) \ldots (W_{\tau_{1} \ldots \tau_{k-1}}W_{\sigma}^{-1}).
\]
From this it suffices to show that $WW_{\tau}^{-1}$ is in the commutator subgroup of $F_{26}$ for any transposition $\tau$.
Write
\begin{align*}
W & = s_{1} \alpha s_{2} \beta s_{3} & W_{\tau} & = s_{1} \beta s_{2} \alpha s_{3}.
\end{align*}
Then we have that
\begin{align*}
WW_{\tau}^{-1} 	& = s_{1} \alpha s_{2} \beta s_{3} s_{3}^{-1} \alpha^{-1} s_{2}^{-1} \beta^{-1} s_{1}^{-1} \\
				& = s_{1}(\alpha s_{2} \beta \alpha^{-1}s_{2}^{-1} \beta^{-1}) s_{1}^{-1} \\
				& = s_{1}(\alpha s_{2} \beta \alpha^{-1} (\beta^{-1} s_{2}^{-1} s_{2} \beta) s_{2}^{-1} \beta^{-1}) s_{1}^{-1} \\
				& = s_{1}(\alpha s_{2} \beta \alpha^{-1} (s_{2}\beta)^{-1} s_{2} \beta s_{2}^{-1} \beta^{-1}) s_{1}^{-1} \\
				& = s_{1} [\alpha, s_{2}\beta][s_{2}, \beta] s_{1}^{-1}.
\end{align*}
Since the commutator subgroup is normal, this computation shows that $WW_{\tau}^{-1}$ is in $[F_{26}, F_{26}]$.
So by our previous logic we see that any anagram relation $W = W_{\sigma}$ is in $[F_{26}, F_{26}]$.
\end{proof}

The next lemma shows how we can reduce the relations in $R(D)$ to a simpler form.
The principle is to use commutators of generators that we know to be in $R(D)$ in order to remove letters from a given anagram relation.

\begin{lemma}\label{lem:reduce}
Suppose the $W_{1} = W_{2}$ is a relation in $R(D)$ where $W_{1}, W_{2}$ are anagrams, and that $\alpha$ is a character appearing in the $W_{i}$.
If we have $[\alpha, \beta] \in R(D)$ for all letters $\beta$ appearing in the $W_{i}$ and we denote by $\hat{W}$ the word $W$ with all instances of the character $\alpha$ removed, we have that the relation $\hat{W}_{1} = \hat{W}_{2}$ is in $R(D)$.
\end{lemma}
\begin{proof}
If we have that $[\alpha, \beta] \in R(D)$ for each character used in the $W_{i}$, then $[\alpha, S] \in R(D)$ for any string $S$ made using those letters (and similarly for any commutator of $\alpha^{\pm1}, S^{\pm1}$).
In particular if we write $W_{1} = S_{1}\alpha S_{2}$, then we see that 
\begin{align*}
[\alpha, S_{1}] W_{1}W_{2}^{-1} 	& = \alpha S_{1} \alpha^{-1} S_{1}^{-1} S_{1} \alpha S_{2} W_{2}^{-1} \\
									& = \alpha S_{1} S_{2} W_{2}^{-1}
\end{align*}
is in $R(D)$.
Repeat this construction to move all instances of $\alpha$ to the start of $W_{1}$, and similarly move all instances of $\alpha^{-1}$ to the end of $W_{2}^{-1}$.
We'll be left with the relation $\alpha^{k} \hat{W}_{1} \hat{W}_{2}^{-1} \alpha^{-k} \in R(D)$, noting that since $W_{1}$ and $W_{2}$ are anagrams $\alpha$ will appear the same number $k$ times in each.
Since $R(D)$ is normal, we thus have that the relation $\hat{W}_{1} = \hat{W}_{2}$ is in $R(D)$, and moreover $\hat{W}_{1}$ and $\hat{W}_{2}$ are anagrams although they may not be words in the dictionary $D$.
\end{proof}

The following consequence of \cref{lem:reduce} describes how new relations can be generated once commutators have been found.

\begin{corollary}\label{cor:combine}
Suppose that we have relations $W_{1} = W_{2}$, $W_{3} = W_{4}$ in $R(D)$, and that after some applications of \cref{lem:reduce} these relations become $\hat{W}_{1} = \hat{W}_{2}$, $\hat{W}_{3} = \hat{W}_{4}$.
If $\hat{W}_{2} = \hat{W}_{3}$, then the relation $\hat{W}_{1} = \hat{W}_{4}$ is in $R(D)$.
\end{corollary}
\begin{proof}
Immediate from multiplying $\hat{W}_{1}\hat{W}_{2}^{-1}$ with $\hat{W}_{3}\hat{W}_{4}^{-1}$.
\end{proof}

\section{Strategy}\label{sec:strategy}

We describe the strategy that we use in our quest to simplify the presentation of $A(D)$.
The basic idea is to iteratively look for commutator relations from our anagrams and then use those relations to reduce and combine our set of anagram relations as in \cref{lem:reduce} and \cref{cor:combine} in the hope of finding more commutators.

\begin{definition}
We say an pair of words $W_{1}, W_{2}$ is an admissible pair if they are anagrams of each other and they are of the form
\begin{align*}
W_{1} & = s_{1}\alpha \beta s_{2} & W_{2} & = s_{1} \beta \alpha s_{2}
\end{align*}
for strings $s_{1}, s_{2}$ and letters $\alpha, \beta$.
Note that if $W_{1}, W_{2}$ are an admissible pair, the relation $W_{1}W_{2}^{-1}$ is a conjugate of the commutator $[\alpha, \beta]$.
\end{definition}

\noindent\textbf{Step 1:} this step is consists of setting up the main data structure we work with. 
For each set of anagrams in $D$ having the same image $\gamma$ in $\Z^{26}$ (we refer to this image as the ``letter count'' of a word), create a complete graph with vertices the words with that letter count.
We think of the edge connecting words $W_{1}$ and $W_{2}$ as the relation $W_{1} = W_{2}$.
We call these graphs the ``anagraphs'' (a portmanteau of anagram and graph, not to be confused with \cite{anagraphs}). \\

\noindent\textbf{Step 2:} search through all the anagraphs we have for any admissible pairs and add them to a running list of known commutators. \\

\noindent\textbf{Step 3:} using the admissible pairs found in the previous step, we reduce and combine our anagraphs using \cref{lem:reduce} and \cref{cor:combine}.
Each anagraph $G_{\gamma}$ corresponds to a word count $\gamma$ in $\Z^{26}$.
If a letter appears in $\gamma$ and our list of admissible pairs tells us that that letter commutes with all others in $\gamma$, then we remove that letter from $\gamma$ and from all the vertices of $G_{\gamma}$.
We let $\gamma'$ be and $G_{\gamma'}$ be the letter count and graph obtained by performing this reduction for each letter of $\gamma$.

This may have the effect of combining two or more anagraphs as in \cref{cor:combine}, since letter counts $\gamma_{1} \neq \gamma_{2}$ may reduce such that $\gamma_{1}' = \gamma_{2}'$.
When this happens we identify the anagraphs $G_{\gamma_{1}'}$ and $G_{\gamma_{2}'}$, identifying those vertices which have reduced to the same string as in \cref{cor:combine}.
In general the reduced anagraphs may end up with several connected components; we always add in edges to ensure that each connected component is complete, which is simply ensuring the transitivity of equality for the group relations as encoded in the anagraphs. \\

\noindent\textbf{Step 4:} return to step $2$ and repeat until no new admissible pairs are found and no more reduction of anagraphs occurs. \\

\noindent\textbf{Step 5:} manually treat the remaining relations. \\

Note that this procedure is not guaranteed to produce useful results; for example with a particularly small dictionary it could be the case that no commutators are found in the second step.
With reasonable English language dictionaries fewer than 10 repetitions of the commutator finding and reduction steps are needed, and the resulting list of a few hundred words can easily be dispatched as most provide no new information.

A Sagemath \cite{sagemath} program written to carry out steps 1 through 4 of this strategy is available through the author's website and GitHub page.

\section{Dictionary specific results}\label{sec:dictionary}

We now describe the anagram group $A(D)$ in the specific case that $D$ is the SOWPODS scrabble word list \cite{sowpods}.

\begin{theorem}\label{thm:sowpods}
The group $A(\textup{SOWPODS})$ has presentation
\[
\langle a, b, \ldots, z \,|\, \textup{all commutators of a pair of generators except the 24 listed below} \rangle
\]
where the missing commutators are:
\begin{itemize}
\item the 6 commutators of each pair of $j, q, x, z$
\item the 5 commutators of $j$ with $f, k, l, w, y$
\item the 6 commutators of $q$ with $b, f, g, k, w, y$
\item the 3 commutators of $x$ with $f, k, v$
\item the 4 commutators of $z$ with $f, k, v, w$
\end{itemize}

\end{theorem}
\begin{proof}
Let $N$ be the normal closure of the set of commutators described in the theorem statement.
We want to show that $N = R(\text{SOWPODS})$.

To show that $N \subseteq R(\text{SOWPODS})$ we must exhibit anagram pairs from the SOWPODS dictionary which realize each of the $301 = 325 - 24$ commutators which generate $N$.
The algorithmic steps of the strategy outlined in \cref{sec:strategy} realizes $271$ of the $325$ possible commutators of generators.
All commutators not found contain at least one of the letters $j, q, x, z$.
The missing commutators are:
\begin{itemize}
\item the 24 exceptional commutators in the theorem statement,
\item all commutators involving the letter $j$ except $[j, a]$, $[j, c]$, and $[j, r]$,
\item all commutators involving the letter $q$.
\end{itemize}
There are $220$ remaining letter counts which are left after the algorithm of \cref{sec:strategy} terminates.
Of these none contain the character $x$, and only two contain the character $z$.
Those containing $z$ (quartziest = quartzites and quartzose = quatorzes) yield no new information, as they are implied by the known commutators: at this point we know that $e, i, s, t$ commute with one another which implies the first relation, and we know that $e, o, r, s, t, z$ commute which implies the second relation.
The remaining $218$ letter counts each contain either $j$ or $q$, and none contains both.
Many of these provide no new relation, for example quickest = quickset is already implied by our knowledge that $e, s, t$ commute.
Using the $271$ commutators we have already established, the other $30$ commutators which are not listed as the $24$ exceptional ones in the theorem statement are found among the remaining anagram groupings, bringing the total list of commutators found to $301 = 325 - 24$.
See \cref{app:data} for anagram pairs which realize each of these $301$ commutators.

To show that $N \supseteq R(\text{SOWPODS})$ we must show that every anagram pair from the SOWPODS dictionary is implied by the relations in $N$.
Our algorithm verifies this for most of the anagrams in the dictionary, as any anagram which reduces to the trivial word after iterations of the second and third steps of the algorithm is implied by the relations in $N$.
For the remaining $220$ letter counts left after the algorithm terminates, one verifies manually that all are implied by the commutators in $N$.
\end{proof}

\begin{remark}\label{rem:raag}
Notably the relations in $R(D)$ are generated only by commutators of generators and not any more involved relations, making $A(\text{SOWPODS})$ a right-angled Artin group.
\end{remark}
\begin{remark}\label{rem:maximal}
One consequence of the fact that $N \supseteq R(\text{SOWPODS})$ is that if $[\alpha, \beta]$ is one of the $24$ exceptional commutators which do not appear in $N$, then all SOWPODS anagrams containing both the characters $\alpha$ and $\beta$ have the same pattern of $\alpha$ and $\beta$ appearing.
So for example there is no anagram pair $W_{1}, W_{2}$ in the SOWPODS dictionary where $W_{1}$ contains a $j$ and then an $x$, while $W_{2}$ contains an $x$ and then a $j$.
So in a sense $A(\text{SOWPODS})$ is maximally abelian, in that any commutator which could possibly arise from anagrams in the SOWPODS dictionary does arise.
\end{remark}

\section{History of the anagram problem}\label{sec:history}

This work was undertaken by the author while he was a graduate student at the University of Chicago, primarily in the spring of 2017.
As it turns out, there is a long and storied history of grad students at the University of Chicago studying anagrammatic groups, stretching back at least to the 1970s.
While following up on references from the article \cite{homophonic}, the author stumbled across \cite{jimmys}.
The article \cite{jimmys} is an account of several whimsical math problems studied by a group of grad students at the University of Chicago in the 1970s.
Apparently there was at that time a tradition of grad students attempting to determine the structure of the group $A$ by hand, with generations of student work being logged in a large paper chart posted on the fourth floor of the mathematics building.
Without computer assistance, several letters had been shown to be in the center of the group, and anagrams exhibiting many of the possible commutator pairs had been filled in on the chart.

The author only learned of this surprising connection in the spring of 2020, at which point the COVID-19 pandemic was in full force.
The author has confirmed that the paper chart mentioned in \cite{jimmys} was still in existence until at the least the early 1990s.
Unfortunately, the closure of the department due to the COVID-19 pandemic and the author's subsequent graduation from the University of Chicago prevented a thorough archaeological exploration of Eckhart Hall to determine if the paper chart has since been lost.

\section{Acknowledgements}\label{sec:acknowledgements}

The author wishes to thank the UChicago crossword crew for encouraging this project, the UChicago pizza seminar for providing a forum where this research was first presented, and the Canada/USA Mathcamp 2020 class ``Grammatical group generation'' for trying to actually learn some math from this project.

\appendix
\addappheadtotoc

\section{SOWPODS anagram pairs realizing commutators}\label{app:data}
Our algorithm takes 5 iterations to terminate on the SOWPODS dictionary.
See \cref{table:rounds} for an overview of how many commutators are found at each step of the algorithm.

\begin{table}[h]
\centering
\begin{tabular}{|c|c|c|}\hline
Iteration & Number of anagraphs & Total number of commutators found \\ \hline
1 & 21640 & 123 \\ \hline
2 & 8992 & 235 \\ \hline
3 & 405 & 266 \\ \hline
4 & 226 & 271 \\ \hline
5 & 220 & 271 \\ \hline
\end{tabular}
\caption{Progress of algorithm on SOWPODS dictionary}\label{table:rounds}
\end{table}

For each of the $301$ commutator relations in our presentation of $A(\text{SOWPODS})$, \cref{table:exhibit} provides an anagram pair realizing that commutator, sorted by which iteration of the strategy they are found at and then lexicographically by $(\alpha, \beta)$.
Note that commutators from previous rounds are used in extracting a commutator from an anagram pair in round 2 and onwards.

\begin{longtable}[h]{|l|ll|ll|}
\caption{Anagrams pairs $(w_{1}, w_{2})$ exhibiting the $301$ commutator relations $[\alpha, \beta]$ in our presentation of the anagram group.}\label{table:exhibit} \\

\hline
Iteration found & $\alpha$ & $\beta$ & $w_{1}$ & $w_{2}$ \\
\hline
\endfirsthead

\hline
Iteration found & $\alpha$ & $\beta$ & $w_{1}$ & $w_{2}$ \\
\hline
\endhead

\hline
\endfoot

\hline
\endlastfoot

1 & a & b & able & bale \\
1 & a & c & acre & care \\
1 & a & d & add & dad \\
1 & a & e & tael & teal \\
1 & a & f & aft & fat \\
1 & a & g & agin & gain \\
1 & a & h & ah & ha \\
1 & a & i & chai & chia \\
1 & a & j & ajwan & jawan \\
1 & a & k & oaky & okay \\
1 & a & l & alps & laps \\
1 & a & m & am & ma \\
1 & a & n & an & na \\
1 & a & o & gaol & goal \\
1 & a & p & apt & pat \\
1 & a & r & arm & ram \\
1 & a & s & asp & sap \\
1 & a & t & eat & eta \\
1 & a & u & daud & duad \\
1 & a & v & ave & vae \\
1 & a & w & awn & wan \\
1 & a & x & coax & coxa \\
1 & a & y & aye & yae \\
1 & a & z & diazin & dizain \\
1 & b & i & carbies & caribes \\
1 & b & l & able & albe \\
1 & b & o & bool & obol \\
1 & b & r & dobra & dorba \\
1 & c & i & cion & icon \\
1 & c & n & acne & ance \\
1 & c & o & cotan & octan \\
1 & c & r & acred & arced \\
1 & c & s & panics & panisc \\
1 & c & y & scye & syce \\
1 & d & e & hide & hied \\
1 & d & g & radge & ragde \\
1 & d & i & di & id \\
1 & d & l & badly & baldy \\
1 & d & o & door & odor \\
1 & d & r & padri & pardi \\
1 & d & u & duo & udo \\
1 & e & f & fief & fife \\
1 & e & g & ego & geo \\
1 & e & h & eh & he \\
1 & e & i & lei & lie \\
1 & e & k & eek & eke \\
1 & e & l & angel & angle \\
1 & e & m & em & me \\
1 & e & n & fiend & fined \\
1 & e & o & reo & roe \\
1 & e & p & creep & crepe \\
1 & e & r & tier & tire \\
1 & e & s & cures & curse \\
1 & e & t & eta & tea \\
1 & e & u & deus & dues \\
1 & e & v & evil & veil \\
1 & e & w & ewe & wee \\
1 & e & y & ley & lye \\
1 & e & z & meze & mzee \\
1 & f & i & defi & deif \\
1 & g & i & algin & align \\
1 & g & k & gingko & ginkgo \\
1 & g & l & bugle & bulge \\
1 & g & n & sign & sing \\
1 & g & o & gore & ogre \\
1 & g & r & segreant & sergeant \\
1 & g & u & rogue & rouge \\
1 & g & y & bogy & boyg \\
1 & h & o & hom & ohm \\
1 & h & s & ahs & ash \\
1 & h & t & baht & bath \\
1 & i & k & sik & ski \\
1 & i & l & gild & glid \\
1 & i & m & aim & ami \\
1 & i & n & intro & nitro \\
1 & i & o & viola & voila \\
1 & i & p & sipe & spie \\
1 & i & r & gird & grid \\
1 & i & s & is & si \\
1 & i & t & its & tis \\
1 & i & v & waiver & wavier \\
1 & i & x & deixes & dexies \\
1 & k & o & kora & okra \\
1 & k & r & chakra & charka \\
1 & k & s & flaks & flask \\
1 & k & u & kue & uke \\
1 & k & y & skyer & syker \\
1 & l & m & alma & amla \\
1 & l & n & alnage & anlage \\
1 & l & o & clod & cold \\
1 & l & s & pulse & pusle \\
1 & l & t & tilted & titled \\
1 & l & u & luna & ulna \\
1 & l & v & delving & devling \\
1 & l & y & idly & idyl \\
1 & m & o & moit & omit \\
1 & m & s & prims & prism \\
1 & m & u & mu & um \\
1 & n & o & mono & moon \\
1 & n & r & unred & urned \\
1 & n & s & mense & mesne \\
1 & n & u & gnu & gun \\
1 & n & y & sny & syn \\
1 & o & p & opt & pot \\
1 & o & r & orc & roc \\
1 & o & s & os & so \\
1 & o & t & pinot & pinto \\
1 & o & u & flour & fluor \\
1 & o & w & tow & two \\
1 & o & y & oy & yo \\
1 & o & z & ozonic & zoonic \\
1 & p & s & cups & cusp \\
1 & p & u & pus & ups \\
1 & p & y & spying & syping \\
1 & r & t & parton & patron \\
1 & r & u & run & urn \\
1 & r & v & larva & lavra \\
1 & r & y & trye & tyre \\
1 & s & t & star & tsar \\
1 & s & u & suer & user \\
1 & s & y & busy & buys \\
1 & t & u & lotus & louts \\
1 & t & y & fluty & fluyt \\ \hline
2 & b & c & bac & cab \\
2 & b & d & bad & dab \\
2 & b & e & bed & deb \\
2 & b & g & bag & gab \\
2 & b & h & boh & hob \\
2 & b & k & bok & kob \\
2 & b & m & bombed & mobbed \\
2 & b & n & ban & nab \\
2 & b & s & abs & sab \\
2 & b & t & bat & tab \\
2 & b & u & babu & buba \\
2 & b & w & bawble & wabble \\
2 & b & y & boy & yob \\
2 & b & z & bozo & zobo \\
2 & c & d & cade & dace \\
2 & c & e & aced & ecad \\
2 & c & f & cafe & face \\
2 & c & g & corgi & orgic \\
2 & c & h & cache & chace \\
2 & c & j & carcajou & carjacou \\
2 & c & l & clay & lacy \\
2 & c & m & came & mace \\
2 & c & p & cap & pac \\
2 & c & t & cate & tace \\
2 & c & u & cur & ruc \\
2 & c & v & cive & vice \\
2 & c & w & cawk & wack \\
2 & d & f & def & fed \\
2 & d & h & dah & had \\
2 & d & k & deke & eked \\
2 & d & m & dam & mad \\
2 & d & n & and & dan \\
2 & d & p & dap & pad \\
2 & d & s & das & sad \\
2 & d & t & dart & trad \\
2 & d & v & avid & diva \\
2 & d & w & daw & wad \\
2 & d & y & dray & yard \\
2 & e & x & taxes & texas \\
2 & f & g & fig & gif \\
2 & f & k & faik & kaif \\
2 & f & l & flit & lift \\
2 & f & n & fain & naif \\
2 & f & r & arf & far \\
2 & f & s & fast & saft \\
2 & f & t & eft & tef \\
2 & f & y & oofy & yoof \\
2 & g & h & gash & hags \\
2 & g & m & gum & mug \\
2 & g & p & gape & page \\
2 & g & s & gas & sag \\
2 & g & t & gat & tag \\
2 & g & v & gave & vega \\
2 & g & w & gowan & wagon \\
2 & h & i & his & ish \\
2 & h & k & hoka & koha \\
2 & h & l & ashler & lasher \\
2 & h & m & amahs & shama \\
2 & h & n & han & nah \\
2 & h & p & hap & pah \\
2 & h & r & hare & rhea \\
2 & h & u & chout & couth \\
2 & h & w & how & who \\
2 & h & y & hay & yah \\
2 & i & u & situs & suits \\
2 & i & y & lily & yill \\
2 & j & r & jar & raj \\
2 & k & l & alko & kola \\
2 & k & m & kam & mak \\
2 & k & n & ken & nek \\
2 & k & p & keep & peek \\
2 & k & t & kat & tak \\
2 & k & v & kavass & vakass \\
2 & k & w & kawa & waka \\
2 & l & p & lap & pal \\
2 & l & r & lear & real \\
2 & l & w & awl & law \\
2 & l & x & axles & laxes \\
2 & l & z & laze & zeal \\
2 & m & n & man & nam \\
2 & m & p & map & pam \\
2 & m & r & mar & ram \\
2 & m & t & mat & tam \\
2 & m & w & mew & wem \\
2 & m & y & may & yam \\
2 & m & z & mozo & zoom \\
2 & n & p & nap & pan \\
2 & n & t & nat & tan \\
2 & n & v & nave & vane \\
2 & n & w & naw & wan \\
2 & n & z & winze & wizen \\
2 & o & v & avo & ova \\
2 & o & x & diaxon & dioxan \\
2 & p & r & par & rap \\
2 & p & t & pat & tap \\
2 & p & v & pavid & vapid \\
2 & p & w & paw & wap \\
2 & r & s & ruse & user \\
2 & r & w & raw & war \\
2 & r & z & bazar & braza \\
2 & s & v & sav & vas \\
2 & s & w & saw & was \\
2 & s & x & axes & saxe \\
2 & s & z & saz & zas \\
2 & t & v & tav & vat \\
2 & t & w & taw & wat \\
2 & t & x & axites & taxies \\
2 & t & z & azote & toaze \\
2 & u & v & uva & vau \\
2 & u & y & guy & yug \\
2 & v & y & nevey & yeven \\
2 & w & y & way & yaw \\ \hline
3 & b & f & beef & feeb \\
3 & b & p & bleep & plebe \\
3 & b & v & adverb & braved \\
3 & c & k & cheek & keech \\
3 & c & x & coexist & exotics \\
3 & d & x & desex & sexed \\
3 & d & z & dozen & zoned \\
3 & f & h & flash & halfs \\
3 & f & m & flamed & malfed \\
3 & f & o & coif & foci \\
3 & f & p & earflap & parafle \\
3 & f & u & furs & surf \\
3 & f & v & favorer & overfar \\
3 & f & w & fretsaw & wafters \\
3 & h & v & halvas & lavash \\
3 & h & x & hoaxed & oxhead \\
3 & h & z & hazmat & matzah \\
3 & i & w & sinew & swine \\
3 & i & z & sitz & zits \\
3 & m & v & mover & vomer \\
3 & m & x & exams & maxes \\
3 & n & x & exons & noxes \\
3 & p & x & expos & poxes \\
3 & p & z & spaz & zaps \\
3 & r & x & rexes & sexer \\
3 & u & w & outwash & washout \\
3 & u & x & exul & luxe \\
3 & u & z & azurines & suzerain \\
3 & v & w & advew & waved \\
3 & w & x & taxwise & waxiest \\
3 & y & z & lysozymes & zymolyses \\ \hline
4 & b & x & bruxes & exurbs \\
4 & c & z & citizen & zincite \\
4 & g & x & exerting & genetrix \\
4 & g & z & gazy & zyga \\
4 & x & y & prexy & pyrex \\ \hline
manual & a & q & aquiline & quiniela \\
manual & b & j & baju & juba \\
manual & c & q & cinque & quince \\
manual & d & j & jumared & mudejar \\
manual & d & q & derequisition & requisitioned \\
manual & e & j & adjuster & readjust \\
manual & e & q & equators & quaestor \\
manual & g & j & gju & jug \\
manual & h & j & hadji & jihad \\
manual & h & q & haiques & quashie \\
manual & i & j & ijtihads & jihadist \\
manual & i & q & uniquest & unquiets \\
manual & j & m & joram & major \\
manual & j & n & abjoints & banjoist \\
manual & j & o & journos & sojourn \\
manual & j & p & jaup & puja \\
manual & j & s & joes & sjoe \\
manual & j & t & jeats & tajes \\
manual & j & u & rejoindures & surrejoined \\
manual & j & v & jayvee & veeyjay \\
manual & l & q & liquidate & qualitied \\
manual & m & q & masque & squame \\
manual & n & q & equinity & inequity \\
manual & o & q & quote & toque \\
manual & p & q & equip & pique \\
manual & q & r & quester & request \\
manual & q & s & quakes & squeak \\
manual & q & t & quoter & torque \\
manual & q & u & maqui & umiaq \\
manual & q & v & quiverer & verquire \\

\end{longtable}

\printbibliography
\end{document}